\documentclass[11pt]{amsart}
\usepackage[margin=1.5in]{geometry}
\usepackage{
 amsmath, 
 amsxtra, 
 amsthm, 
 amssymb, 
 etex, 
 mathrsfs, 
 mathtools, 
 tikz-cd, 
 bbm,
 xr,
 comment,
 enumitem}
\usepackage[toc]{appendix} 
\usepackage[all]{xy}
\usepackage{hyperref}
\usepackage{url}
\usepackage{tipa}
\usepackage{dsfont}
\usepackage{ textcomp }
\usepackage{stmaryrd} 
\usepackage{amssymb}
\usepackage{mathabx} 
\usepackage{amsmath} 
\usepackage{amscd}
\usepackage{amsbsy}
\usepackage{comment, enumerate}
\usepackage{color}
\usepackage{mathtools,caption}
\usepackage{tikz-cd}
\usepackage{longtable}
\usepackage[utf8]{inputenc}
\usepackage[OT2,T1]{fontenc}
\usepackage{changepage}
\usepackage{commath,enumerate}
\usetikzlibrary{backgrounds}

\DeclareMathOperator{\Gal}{Gal}

\DeclareMathOperator{\ord}{ord}

\DeclareMathOperator{\GL}{GL}

\newtheorem{theorem}{Theorem}[section]
\newtheorem*{theorem*}{Theorem}
\newtheorem{lemma}[theorem]{Lemma}

\newtheorem{proposition}[theorem]{Proposition}
\newtheorem{corollary}[theorem]{Corollary}

\newtheorem{defn}[theorem]{Definition}

\numberwithin{equation}{section}
\newtheorem{lthm}{Theorem} 
\newtheorem{corl}[lthm]{Corollary}

\theoremstyle{remark}
\newtheorem{remark}[theorem]{Remark}
\newtheorem{example}[theorem]{Example}

\newcommand{\EE}{\mathbb{E}}
\newcommand{\R}{\mathbb{R}}
\newcommand\EatDot[1]{}

\newcommand{\cX}{\mathcal{X}}
\newcommand{\QQ}{\mathbb{Q}}
\newcommand{\ZZ}{\mathbb{Z}}
\newcommand{\FF}{\mathbb{F}}
\newcommand{\Qp}{\mathbb{Q}_p}
\newcommand{\Zp}{\mathbb{Z}_p}

\definecolor{Green}{rgb}{0.0, 0.5, 0.0}

\newcommand{\cC}{\mathcal{C}}

\newcommand{\Z}{\mathbb{Z}}
\newcommand{\fm}{\mathfrak{m}}

\newcommand{\Deck}{\mathrm{Deck}}

\renewcommand{\t}{\mathfrak t}

\renewcommand{\ss}{\mathrm{ss}}
\newcommand{\cY}{\mathcal{Y}}
\newcommand{\Div}{\textup{Div}}
\newcommand{\Pic}{\textup{Pic}}

\newcommand{\VV}{\mathbb{V}}

  \DeclareFontFamily{U}{wncy}{}
  \DeclareFontShape{U}{wncy}{m}{n}{<->wncyr10}{}
  \DeclareSymbolFont{mcy}{U}{wncy}{m}{n}
  \DeclareMathSymbol{\sha}{\mathord}{mcy}{"58}
  \DeclareMathSymbol{\zhe}{\mathord}{mcy}{"11}

\mathcode`l="8000
\begingroup
\makeatletter
\lccode`\~=`\l
\DeclareMathSymbol{\lsb@l}{\mathalpha}{letters}{`l}
\lowercase{\gdef~{\ifnum\the\mathgroup=\m@ne \ell \else \lsb@l \fi}}%
\endgroup

\makeatletter
\newcommand{\mylabel}[2]{#2\def\@currentlabel{#2}\label{#1}}
\makeatother
  
\title[Constant $\Zp$-towers of graph coverings]{On $\Zp$-towers of graph coverings arising from a constant voltage assignment}

\makeatletter
\let\@wraptoccontribs\wraptoccontribs
\makeatother

\usepackage{tikz,xcolor,hyperref}

\author[A. Lei]{Antonio Lei}
\address[Lei]{Department of Mathematics and Statistics\\University of Ottawa\\
150 Louis-Pasteur Pvt\\
Ottawa, ON\\
Canada K1N 6N5}
\email{antonio.lei@uottawa.ca}

\author[K. Müller]{Katharina Müller}
\address[Müller]{Institut für Theoretische Informatik, Mathematik und Operations Research, Universität der Bundeswehr München, Werner-Heisenberg-Weg 39, 85577 Neubiberg, Germany}
\email{katharina.mueller@unibw.de}

\begin{document}

\begin{abstract}
    
We investigate properties of $\Zp$-towers of graph coverings that arise from a constant voltage assignment. We prove the existence and uniqueness (up to isomorphisms) of such towers. Furthermore, we study the Iwasawa invariants of these towers, and apply our results to towers of isogeny graphs enhanced with level structures, as well as towers arising from volcano graphs.
\end{abstract}

\maketitle
\section{Introduction}
\subsection{Background}
Let $p$ be a fixed prime number. We write $\Zp$ for the ring of $p$-adic integers. Vallières, Gonet, and McGown--Vallières initiated the study of $\Zp$-towers of graph coverings (see \cite{vallieres,vallieres2,vallieres3,gonet-thesis,gonet22}). Let $X$ be a finite connected undirected graph. A $\Zp$-tower of graph coverings $(X_n)_{n\ge0}$ over $X$ is a series of graph coverings $X_n/X$, where  $X_n$ is a connected undirected graph such that $X_n/X$ is a Galois covering whose Galois group is isomorphic to the cyclic group $\ZZ/p^n\ZZ$. We write $\kappa_n$ for the number of spanning trees of $X_n$. Then there exist integers $\mu$, $\lambda$ and $\nu$ such that
\begin{equation}
\label{eq:Iw}\ord_p(\kappa_n)=\mu p^n+\lambda n+\nu    
\end{equation}
for $n\gg0$. This is analogous to the classical result of Iwasawa on class numbers of finite sub-extensions inside a $\Zp$-extension a number field \cite{iwasawa73}.

 A  $\Zp$-tower of graph coverings over an undirected $X$ can be constructed using a $\Zp$-valued voltage assignment on a \textit{directed} graph (after choosing an orientation for each edge of $X$). The graph $X_n$ can be obtained from the derived graph attached to a $\ZZ/p^n\ZZ$-valued voltage assignment, which is a priori a directed graph (see Remark~\ref{rk:undirected} for a detailed discussion). 

In the present article, we study $\Zp$-towers of graph coverings over  \textit{directed} graphs. We say that a directed graph is connected if its image under the forgetful map (which is the natural map that sends a directed edge to an undirected edge) is connected as an undirected graph. We are interested in $\Zp$-towers that arise from a voltage assignment on a finite connected directed graph whose image consists of a single element (i.e. a constant voltage assignment).

The primary motivation for studying such \( \mathbb{Z}_p \)-towers is that they naturally occur as isogeny graphs enhanced with level structures, as demonstrated in \cite{LM2}. We briefly recall the definition of these graphs below.

Let $p$ and $r$ be two distinct prime numbers. 
Let $N$ be a natural number coprime to $rp$ and let $\ell$ be a prime such that $\ell\equiv 1\pmod{Np}$. 
 We fix a finite field $\FF_q$, where $q$ is an even power of $r$. {We say that two elliptic curves defined over $\FF_q$ are equivalent if they are isomorphic over $\overline{\FF_q}$.} We fix a set of representatives of the equivalence classes of elliptic curves defined over $\mathbb{F}_{q}$, which we denote by $S$. 
Given two elliptic curves $E,E'\in S$, an isogeny $\phi:E\to E'$ is a non-constant morphism of varieties that sends the point at infinity of $E$ to the point at infinity of $E'$. We say that $\phi$ is an $\ell$-isogeny if its degree is equal to $l$. Since $l\nmid q$, this is equivalent to $\ker(\phi)$ containing $l$ elements in $E(\overline{\FF_q})$.

\begin{defn}\label{def:intro}
     For an integer $n\ge0$, we define the directed graph $X_l^q(Np^n)$ whose vertices are given by triples $(E,Q_1,Q_2)$, where $E\in S$ and $\{Q_1,Q_2\}$ is a basis of $E[Np^n]\cong (\ZZ/Np^n\ZZ)^2$. The edges of $X_l^q(Np^n)$ are given by $l$-isogenies, i.e. the set of edges from $(E,Q_1,Q_2)$ to $(E',Q_1',Q_2')$ is the set of $l$-isogenies $\phi\colon E\to E'$ such that $\phi(Q_1)=Q_1'$ and $\phi(Q_2)=Q_2'$.

We define the directed graph $Y_l^q(p^nN)$ as follows.
    The set of vertices of $Y_l^q(Np^n)$ is given by tuples $(E,R_1,R_2,\zeta)$, where $E\in S$, $\{R_1,R_2\}$ is a basis of $E[N]$ and $\zeta$ is a primitive $p^n$-th root of unity. Note that we can write a point $Q_i$ of order $Np^n$ as $Q_i=R_i+P_i$, where $R_i$ is of order $N$ and $P_i$ is of order $p^n$-torsion. Let $\Psi_n$ be the following surjection
\begin{align*}
    \Psi_n\colon V(X_l^q(Np^n))&\to V(Y_l^q(Np^n)), \\ 
    (E,Q_1,Q_2)&\mapsto (E,R_1,R_2,\langle P_1,P_2\rangle),
\end{align*}
where $\langle-,-\rangle$ denotes the Weil pairing.
For every vertex $v\in V(Y_l^q(Np^n))$, we fix a pre-image $v'\in V(X_l^q(Np^n))$ under $\Psi_n$. 
 The edges of  $Y_l^q(Np^n)$ are given as follows: 
 for every edge from $v'$ to some $w''\in \Psi_n^{-1}(w)$ in $X_l^q(Np^n)$, we draw an edge from $v$ to $w$ in $Y_l^q(Np^n)$.
     
\end{defn}

Let $Y^{\ss}$ be the supersingular subgraph of $Y_l^q(Np)$. For $n\ge1$, define $Y^{\ss}_n$ as the pre-image of $Y^{\ss}$ in $Y_l^q(Np^n)$ under the natural projection $Y_l^q(Np^n)\to Y_l^q(Np)$. We show in Lemma~\ref{lem:constant} that the collection of graphs $(Y_n^\ss)_{n\ge1}$ gives rise to a $\Zp$-tower over $Y^\ss$ that is defined by a constant voltage assignment. This motivates us to develop a general framework for studying arithmetic properties of such $\Zp$-towers.

\subsection{Main results}

We begin with a theorem that gives a sufficient and necessary condition for the existence of $\Zp$-towers over a graph $X$ given by a constant voltage assignment. Furthermore, we show that it is unique up to isomorphisms of graphs. The proof is straightforward and builds on established results from the literature. Given a directed graph $X$, let $\tilde X$ be the directed graph, where we add a supplementary edge $e^\iota$ to $X$ for each edge $e$ of $X$. Given a path $P$ in $\tilde X$, we define its \textit{weight} to be the number of edges of $P$ that belong to $X$ minus the number of supplementary edges in $P$ (see Definition~\ref{def:weight}). 

\begin{lthm}[Corollaries~\ref{cor:exist}, \ref{cor:unique} and \ref{cor:stabilize}]\label{thmA}
    Let $X$ be a finite connected directed graph. Let $\alpha$ be a constant $\Zp$-valued voltage assignment on $X$ whose image is an element of $\Zp^\times$. 
    \begin{itemize}
        \item[(a)] The derived graphs of $\alpha$ modulo $p^n$ give rise to a $\Zp$-tower of graph coverings if and only if $\tilde X$ admits a directed cycle that is of weight coprime to $p$. Furthermore, the $\Zp$-tower is independent of $\alpha$, up to isomorphism of graphs.
        \item[(b)] If $\tilde X$ contains a cycle, the derived graphs of $\alpha$ modulo $p^n$ give rise to a disjoint union of isomorphic $\Zp$-towers over $X$.
    \end{itemize}
\end{lthm}

In particular, given a finite connected directed graph that contains a cycle, Theorem~\ref{thmA} tells us that there is a \textit{canonical} $\Zp$-tower over $X$, to which we shall refer as the \textit{\textbf{constant $\Zp$-tower}} over $X$. After applying the forgetful map, we can regard this as a $\Zp$-tower of undirected graphs. The Iwasawa invariants of this tower as given in \eqref{eq:Iw} will be denoted by $\mu(X)$ and $\lambda(X)$. We call them the $\mu$-invariant and the $\lambda$-invariant of $X$.

The following theorem gives sufficient conditions for the (non-)vanishing of the $\mu$-invariant of a constant $\Zp$-tower.

\begin{lthm}[Theorems~\ref{thm:mu>0} and \ref{thm:mu=0}]\label{thmB}
     Let $X$ be a finite connected directed graph that contains a cycle.
     \begin{itemize}
         \item[(a)] Suppose that $p$ divides the in-degree and out-degree of each vertex of $X$, then $\mu(X)>0$.
         \item[(b)] Suppose that the total degree of each vertex (i.e. the sum of the in-degree and out-degree) is equal to some constant $k$ that is coprime to $p$. If the adjacency matrix of $X$ is normal, then $\mu(X)=0$.
     \end{itemize}
\end{lthm}

Let $X^{\ss}$ be the supersingular subgraph of $X_l^q(Np)$ given by Definition~\ref{def:intro} and define $X^{\ss}_n$ as its pre-image in $X_l^q(Np^n)$. We recall from \cite[Corollary~4.7]{LM2} that there exists a {minimal} integer $m_0$ such that the number of connected components of $X_n^\ss$ stabilizes. Let $Z_{m_0}$ be a connected component of $X^{\ss}_{m_0}$. For $n\ge m_0$, let $Z_n$ be the pre-image of $Z_{m_0}$ in $X_n$, then $Z_n/Z_{m_0}$ is Galois. Furthermore, the inverse limit $G:=\displaystyle\varprojlim_{n\ge m_0} \Gal(Z_n/Z_{m_0})$ is an open subgroup of $\GL_2(\Zp)$.


Let \[H_n=\left\{x\in \textup{SL}_2(\ZZ/p^nN\ZZ)\mid x\equiv I \pmod {Np}\right\}\] and $Z'_n=(Z_n)^{H_n}$. We prove that $(Z'_n)_{n\ge m_0}$ is a constant $\Z_p$-tower over $Z'_{m_0}$. Furthermore, it can be realized as an intermediate tower of graph coverings of $(Z_n)_{n\ge m_0}$. This allows us to formulate a $\mathfrak{M}_H(G)$-conjecture for $(Z_n)_{n\ge m_0}$ analogously to the counterpart in non-commutative Iwasawa theory as studied in \cite{CFKSV} (see Section \ref{S:isogeny} for details). This conjecture is of interest since it implies that the $K$-theoretic formulation of the non-commutative Iwasawa main conjecture is true for the tower $(Z_n)_{n\ge m_0}$ (as demonstrated in \cite{KM}). In the present article, we prove that the $\mathfrak{M}_H(G)$-conjecture holds for $(Z_n)_{n\ge m_0}$:

\begin{corl}[Corollaries~\ref{cor:isogeny} and \ref{cor:isogeny-regular}]\label{corC}
Assume that $p>2$.
    Let $(Z'_n)_{n\ge m_0}$ be the constant $\Z_p$-tower over $Z'_{m_0}$. Then $\mu(Z'_{m_0})=0$. Furthermore, the $\mathfrak{M}_H(G)$-conjecture holds for the tower $(Z_n)_{n\ge m_0}$. If in addition, the number of spanning trees of $Z'_{m_0}$ {and the number of vertices of $Z'_{m_0}$} is not divisible by $p$, then $\lambda(Z'_{m_0})=1$. 
\end{corl}

Theorem \ref{thmB} only proves the vanishing of $\mu(X)$ for graphs where the degree of each vertex is equal to the same constant $k$. Under mild additional assumptions, we prove that the $\mu$-invariant  vanishes and the $\lambda$-invariant is one for balanced directed graphs (i.e. for each vertex, the in-degree and out-degree are equal).
\begin{lthm}[Theorem~\ref{thm:regular-balanced}]\label{thmD}
    Let $X$ be a finite directed and balanced graph. Suppose that $p>2$ or ${\tilde{X}}$ contains a cycle whose {weight} is coprime to $p$. Let $k$ be the sum of the in-degree (or equivalently out-degree) of all the vertices of $X$. Let $\kappa_X$ denote the number of spanning trees in the image of $X$ under the forgetful map. If $p\not \mid k\kappa_X$, then $\mu(X)=0$ and $\lambda(X)=1$.
\end{lthm}
The last assertion of Corollary~\ref{corC} is proved using Theorem~\ref{thmD}.

Readers familiar with the classical Iwasawa theory of ideal class groups of number fields may find it interesting that the condition $p\not\mid k\kappa_X$ resembles the definition of a regular prime. Indeed, $p$ is said to be regular if it does not divide the class number of $\QQ(\mu_p)$. When $p$ is a regular prime, the Iwasawa invariants of the cyclotomic $\Zp$-extension of $\QQ(\mu_p)$ are trivial. The $\lambda$-invariant of a balanced graph  {$X$ such that $\tilde{X}$ contains a cycle of weight} coprime to $p$ is at least one (see Remark~\ref{rk:lambda1}). Therefore, the analogue of the regularity condition in the context of graph coverings (that is, $p\nmid k\kappa_X$) guarantees that the Iwasawa invariants are minimum.

The last family of graphs we study in this article are the so-called volcano graphs, which are neither regular nor balanced. We study their constant $\Zp$-tower coverings and compute the corresponding Iwasawa invariants. These graphs are used to describe the isogeny graphs defined by ordinary elliptic curves, as shown in \cite{kohel}. 

\begin{lthm}[Corollaries~\ref{cor:volcano} and \ref{cor:augmented}]\label{thmE}
    Let $X$ be a finite volcano graph. 
    \begin{itemize}
        \item[(a)] If the crater of $X$ is a cycle, then $\mu(X)=0$ and $\lambda(X)=1$.
        \item[(b)] If the crater of $X$ is a single vertex together with two loops, then $\mu(X)=v_p(2)$ and $\lambda(X)=1$. 
    \end{itemize}
\end{lthm}
The reader is referred to \S\ref{S:volcano} for a reminder of the definitions of a volcano graph and its crater. We remark that the proof of Theorem~\ref{thmE} relies on an explicit count of spanning trees in a constant $\Zp$-tower, rather than relying on Theorem~\ref{thmD}. Nonetheless, Theorem~\ref{thmD} can be applied to certain isogeny graphs of ordinary elliptic curves, which are closely related to volcano graphs; see Remark~\ref{rk:ordinary} for a detailed discussion.

\section*{Acknowledgements}
The research of AL is supported by the NSERC Discovery Grants Program RGPIN-2020-04259 and RGPAS-2020-00096. The authors thank Andreas Nickel, Mateja Šajna and Daniel Vallières for helpful and fruitful discussions on topics related to this article. 

\section{Notions in graph theory}

We recall notions in graph theory that will be used throughout the article. Given a (directed or undirected) graph $X$ (multiple edges and loops are allowed), we write $\VV(X)$ and $\EE(X)$ for the set of vertices and edges of $X$, respectively. 
If $X$ is a directed graph, given $v\in \VV(X)$, we write $d_o(v)$ and $d_i(v)$ for the out-degree and the in-degree of $v$, respectively. In the case of undirected graph, we write $\deg(v)$ for the degree of a vertex.

\begin{defn}
    Let $X$ be an undirected graph (finite or infinite) such that the degree of every vertex is finite. For every vertex $v$ we define the \textbf{principal divisor}
    \[P_v=\deg(v)v-\sum_{e\in \EE_v}t(e),\]
    where $\EE_v$ is the set of edges adjacent to $v$ and $t(e)$ is the second vertex reached by the edge $e$.

    We define the \textbf{divisor group} of $X$ by
    \[
        \textup{Div}(X)=\left\{\sum_{v\in \VV(X)}a_v v:a_v\in \ZZ, a_v=0 \textup{ for all but finitely many $v$}\right\}.\]
        The group of principal divisors is the subgroup of $\Div(X)$ generated by $P_v$ as $v$ runs through $v\in\VV(X)$: 
        \[\textup{Pr}(X)=\langle P_v: v\in \VV(X)\rangle \subseteq \textup{Div}(X).\]
        The \textbf{Picard group} of $X$ is defined as the quotient
       \[ \textup{Pic}(X)=\textup{Div}(X)/\textup{Pr}(X).\]
       
       If $X$ is a directed graph, we define these objects correspondingly after applying the forgetful map to $X$.
\end{defn}
We recall that the number of spanning trees in a finite connected undirected graph is equal to the cardinality of the torsion subgroups of the Picard group.

\begin{defn}
    If $Y/X$ is a covering of directed graphs with the projection map $\pi:Y\to X$, the group of \textbf{deck transformations} of $Y/X$, denoted by $\Deck(Y/X)$, is the group of graph automorphisms $\sigma:Y\to Y$ such that $\pi\circ\sigma=\pi$. %
    
    We say that $Y/X$ is \textbf{$d$-sheeted} if $d$ is a positive integer such that each element of $V(X)$ has $d$ pre-images in $V(Y)$.
    
Finally, We say that the covering $Y/X$ is \textbf{Galois}  if it is a $d$-sheeted covering and $\vert \textup{Deck}(Y/X)\vert =d$.
\end{defn}

Galois coverings can be constructed using voltage assignments, which we recall below.

\begin{defn}
    Let $X$ be a directed graph and $(G,\cdot)$ a group. A $G$-valued \textbf{voltage assignment} on $X$ is a function $\alpha:\EE(X)\rightarrow G$. 

    Given a $G$-valued voltage assignment on $X$, we define the derived graph $X(G,\alpha)$ whose vertices and edges are given by $\VV(X)\times G$ and $\EE(X)\times G$, respectively. {If $e\in\EE(X)$ is an edge of $X$ going from $s$ to $t$, then the edge $(e,\sigma)\in\EE(X(G,\alpha))= \EE(X)\times G$ goes from $(s,\sigma)$ to $(t,\sigma\cdot\alpha(e))$.}

\end{defn}

The derived graph $X(G,\alpha)$ naturally gives rise to a covering of $X$ (defined by $Y\rightarrow X$, $(x,\sigma)\mapsto x$).

\begin{lemma}\label{lem:iso}
    Let $X$ be a directed graph and $(G,\cdot)$ a group. Let $\Phi$ be an automorphism on $G$ and $\alpha$ a $G$-valued voltage assignment on $X$. Then $X(G,\alpha)$ and $X(G,\Phi\circ\alpha)$ are isomorphic direct graphs.
\end{lemma}
\begin{proof}
    The automorphism $\Phi$ defines a natural map $X(G,\alpha)\rightarrow X(G,\Phi\circ\alpha) $ given by $(x,\sigma)\mapsto (x,\Phi(\sigma))$. A direct computation shows that this is an isomorphism of graphs.
\end{proof}

\begin{remark}
Suppose that $\alpha$ is a $G$-valued voltage assignment on a finite connected directed graph $X$ and the derived graph $X(G,\alpha)$ is connected, then $Y/X$ is Galois, whose Galois group is isomorphic to $G$ (see \cite[Theorem~2.10]{gonet22} and \cite[\S2.3]{gonet-thesis}).    
\end{remark}

\begin{defn}
    Let $X$ be a directed graph. A \textbf{$\Zp$-covering} of $X$, or a \textbf{$\Zp$-tower} over $X$, is a sequence of graph coverings
\[
X=X_0\leftarrow X_1\leftarrow X_2\leftarrow \cdots \leftarrow X_n\leftarrow\cdots
\]
such that for each $n\ge0$, the cover $X_n/X$ is Galois with Galois group isomorphic to $\ZZ/p^n\ZZ$.
\end{defn}

\begin{remark}\label{rk:undirected}
    In \cite{gonet-thesis,gonet22,vallieres,vallieres2,vallieres3}, a $\Zp$-covering of an undirected graph $X$ is constructed as follows. Let $\tilde X$ be the directed graph such that $\VV(\tilde X)=\VV(X)$ and we define $\EE(\tilde X)$ to be the set of edges obtained from assigning each edge $e$ of $X$ to two oriented edges $e_1,e_2$, joining the same vertices as $e$, with $e_1$ and $e_2$ going in opposite directions. Let $\alpha$ be a voltage assignment on $\tilde X$ such that $\alpha(e_1)=\alpha(e_2)^{-1}$ for all $e\in\EE(X)$. In the derived graph $\tilde X(G,\alpha)$, the edges come in pairs, going into opposite directions. We obtain an undirected graph   after identifying such two edges with a single undirected edge.
    
    Alternatively, we can obtain a directed graph  $X'$ by assigning each edge of $X$ a direction and define a $G$-valued voltage assignment $\alpha'$ on $X'$. We obtain an undirected graph after applying the forgetful to $X'(G,\alpha')$.

    A $\Zp$-valued voltage assignment results in $\ZZ/p^n\ZZ$-valued assignments for all integer $n\ge0$. The procedures described above give rise to a sequence of coverings of undirected graphs 
    $$X\leftarrow X_1\leftarrow X_2\leftarrow\cdots X_n\leftarrow \cdots.$$
    If each $X_n$ is connected, this is a $\Zp$-tower over $X$. We refer the reader to \cite[Corollary~2.16]{DLRV} for a criterion of the connectedness condition.

    We can calculate the $\mu$- and $\lambda$-invariants in \eqref{eq:Iw} by an explicit power series. Let $\{v_1,\dots v_r\}$ be the set of vertices of $X$. Let $\alpha$ be a $\Zp$-valued assignment on $\tilde X$. Define
    \[
    M_{X,\alpha}(T):=\det(D-A_\alpha)\in\Zp[[T]]
    \]
    where $D$ denotes the valency matrix of $X$ and $A_\alpha$ is the $r\times r$ matrix whose $(i,j)$-entry is given by $\sum_e (1+T)^{\alpha(e)}$, where the sum runs over all edges in $\tilde X$ that go from $v_i$ to $v_j$. When $T=0$, $A_\alpha$ becomes the adjacency matrix of $X$ and so $M_{X,\alpha}(0)$ is the determinant of the Laplacian of $X$, which is zero. In particular, $T$ divides $M_{X,\alpha}(T)$. The $p$-adic Weierstrass preparation theorem implies that
    $$
M_{X,\alpha}(T)=p^m Tg(T)u(T),
    $$
    for some integer $m\ge 0$, a distinguished polynomial $g(T)$  and $u(T)\in\Zp[[T]]^\times$. The invariants in \eqref{eq:Iw} satisfy $\mu=m$ and $\lambda=\deg(g)$.
\end{remark}

For the rest of the article, we will only consider directed graphs that are finite and connected.

\begin{lemma}\label{lem:connected}
      Let $X$ be a finite connected directed graph. Let $(X_n)_{n\ge0}$ be a $\Zp$-tower over $X$. Then $X_n$ is connected for all $n$.
\end{lemma}
\begin{proof}
    Suppose $X_n$ contains $r$ connected components. It follows from \cite[proof of Corollary~2.5]{LM2} that the connected components are isomorphic to each other. In particular, $\Gal(X_n/X)$ admits a subgroup of order $r!$ (given by the permutations of the connected components). Since $\Gal(X_n/X)\simeq\ZZ/p^n\ZZ$, this is only possible if $p=r=2$. Suppose that $p=r=2$. Then $X_n\simeq Y_n\bigsqcup Y_n$, where $Y_n$ is a connected graph, and $Y_n/X$ is a graph covering of degree $2^{n-1}$. In particular, the group of Deck transformations of $X_n/X$ is isomorphic to $\Deck(Y_n/X)\times\ZZ/2\ZZ$, which is not cyclic whenever $n\ge2$. Thus, $r=1$. If $X_n$ is connected for $n\ge2$, then $X_1$ is also connected since $X_n$ is a covering of $X_1$.
\end{proof}

\section{Derived graphs arising from a constant voltage assignment}

The goal of this section is to prove Theorem~\ref{thmA}. We begin with the following definitions that will be used throughout the section.

\begin{defn}
Let $\alpha$ be a $\Zp$-valued voltage assignment on a graph $X$. Given an integer $n\ge1$, we write $\alpha_{/n}$ for the voltage assignment given by the composition of $\alpha$ with the projection map $\Zp\rightarrow \ZZ/p^n\ZZ$.

\end{defn}

\begin{defn}
Let $X$ be a finite connected directed graph. We say that a $\Zp$-tower $(X_n)_{n\ge0}$ over $X$ is a \textbf{constant $\Zp$-tower} if $X_n=X(\ZZ/p^n\ZZ,\alpha_{/n})$ for some constant $\Zp$-valued voltage assignment $\alpha$ on $X$ (i.e. $\alpha(e)$ takes the same value for all $e\in\EE(X)$). We shall say that the common value $\alpha(e)$ is the \textbf{parameter} attached to the cover. If no confusion may arise, we shall identify $\alpha$ with the common value of $\alpha(e)$.
\end{defn}

\begin{lemma}\label{lem:connected-unit}
    Let $X$ be a finite connected directed graph. Let $(X_n)_{n\ge0}$ be a constant $\Zp$-tower over $X$, with parameter $\alpha$. Then $\alpha$ belongs to $\Zp^\times$.
\end{lemma}
\begin{proof}
    Suppose that $\alpha\in p\ZZ$. Then $\alpha_{/1}$ sends all $e\in\EE(X)$ to $0$. Therefore, any vertex $(v,a)\in \VV(X_1)=\VV(X)\times\ZZ/p\ZZ$ only admits edges to other vertices of the form $(w,a)$. In particular, there is no path from $(v,a)$ to $(w,b)$ if $a\ne b$. Thus, $X_1$ is not connected, which contradicts Lemma~\ref{lem:connected}.
\end{proof}
\begin{defn}\label{def:weight}
    Let $X$ be a finite directed graph and let $\tilde X$ be the directed graph such that $\VV(X)=\VV(\tilde X)$ and $\EE(\tilde X)=\EE(X)\bigsqcup \EE(X)^\iota$, where $\EE(X)^\iota$ is the set of edges obtained from reversing the direction of the edges in $\EE(X)$. For $e\in \EE(\tilde X)$, we write
    \[w(e)=\begin{cases} 1 & e\in \EE(X),\\
    -1& e\in\EE(X)^\iota.\end{cases}\]
    If $P$ be a  directed path in $\tilde X$ consisting of the edges $e_1,\dots, e_r$, the \textbf{weight} of $P$ is defined to be $w(P)=\sum_{i=1}^r w(e_i)$. We say that $X$ admits a cycle of weight $t$ if $\tilde X$ contains a cycle of weight $t$. 
\end{defn}

\begin{proposition}\label{prop:connected}
    Let $X$ be a finite connected directed graph. Let $\alpha$ be a constant $\Zp$-valued voltage assignment on $X$ that takes values in $\Zp^\times$. Then $X(\ZZ/p^n\ZZ,\alpha_{/n})$ is connected for all $n$ if and only if $X$ admits a cycle of weight coprime to $p$.
\end{proposition}
\begin{proof}
Let $X_n=X(\ZZ/p^n\ZZ,\alpha_{/n})$. Let $\alpha_n\in\ZZ/p^n\ZZ$ denote the value taken by $\alpha_{/n}$. The connectedness of $X_n$ is equivalent to that of $\tilde X_n$, where $\tilde X_n$ is the directed graph defined as in Defintiion~\ref{def:weight}. Let $\cC$ be the set of cycles in $\tilde X$. For each $C\in\cC$, we write $r_C$ for the weight of $C$. It follows from \cite[Theorem~4]{gonet-thesis} that $\tilde X_n$ is connected if and only if the subgroup of $\ZZ/p^n\ZZ$ generated by $\{\alpha_n\cdot r_C: C\in\cC\} $ equals $\ZZ/p^n\ZZ$. As $\alpha_n$ is coprime to $p$, this is the case if and only if $r_C$ is coprime to $p$ for some $C$.
\end{proof}

\begin{corollary}\label{cor:exist}
    Let $X$ be a finite connected directed graph. A constant $\Zp$-valued voltage assignment $\alpha$ on $X$ gives rise to a $\Zp$-tower of graph coverings if and only if $\alpha$ takes values in $\Zp^\times$ and $X$ admits a cycle that is of weight coprime to $p$.
\end{corollary}
\begin{proof}
    This follows from combining Lemma~\ref{lem:connected} with Lemma~\ref{lem:connected-unit} and Proposition~\ref{prop:connected}.
\end{proof}
This proves the existence asserted in part (a)  of Theorem~\ref{thmA}. The uniqueness is given by the following:

\begin{corollary}\label{cor:unique}
        Let $X$ be a finite connected directed graph. Let $(X_n)_{n\ge0}$ and $(X'_n)_{n\ge0}$ be constant $\Zp$-towers of connected graphs over $X$. Then $X_n$ and $X'_n$ are isomorphic for all $n\ge0$.
\end{corollary}
\begin{proof}
    Let $\alpha$ and $\alpha'$ be the corresponding parameters. Lemma~\ref{lem:connected-unit} says that $\alpha,\alpha'$ belong to $\Zp^\times$. In particular, there is a group isomorphism $\Phi:\Zp\rightarrow\Zp$ such that $\alpha'=\Phi\circ\alpha$ given by the multiplication by a fixed element of $\Zp^\times$. Hence, the corollary is a consequence of Lemma~\ref{lem:iso}. 
\end{proof}
In particular, without loss of generality, we may assume that the parameter of a constant voltage assignment is equal to $1$.

Suppose that $X$ is a bouquet. In other words, $|\VV(X)|=1$ and $\EE(X)\ne \emptyset$. Then any $\Zp$-voltage assignment whose image contains an element that belongs to $\Zp^\times$ gives rise to a $\Zp$-tower of connected graphs over $X$ (see \cite[Corollary~2.16]{DLRV}). Indeed, as a bouquet contains a cycle of length one, Proposition~\ref{prop:connected} applies.

We define the following power series associated with the constant voltage assignment whose parameter is equal to $1$, given as in Remark~\ref{rk:undirected}:
\begin{defn}\label{def:MXT}
    Let $X$ be a finite connected directed graph. We define the power series
    \[
    M_X(T)=\det(D-A(1+T)-A^t(1+T)^{-1})\in\Zp[[T]],
    \]
    where $D$ is the diagonal matrix whose diagonal entries are given by the sum of the in-degree and out-degree of the vertices of $X$ and $A$ is the adjacency matrix of $X$.
\end{defn}

\begin{remark}\label{rk:mu-lambda}
Recall that $M_X(0)=0$ and we can write 
\[
M_X(T)=p^mTg(T)u(T)
\]
 for some integer $m\ge 0$, a distinguished polynomial $g(T)$  and $u(T)\in\Zp[[T]]^\times$. If $X$ satisfies the conditions given in Corollary~\ref{cor:exist}, we have $\mu(X)=m$ and $\lambda(X)=\deg(g)$. 
 \end{remark}

We now turn our attention to the proof of Theorem~\ref{thmA}(b).
\begin{proposition}\label{prop:stabilize}
    Let $X$ be a finite connected directed graph such that $\tilde X$ contains a cycle. Let $\alpha$ be a constant $\Zp$-valued voltage assignment on $X$ that takes values in $\Zp^\times$. The number of connected components $r_n$ of $X_n=X(\ZZ/p^n,\alpha_{/n})$ is uniformly bounded. Let $n_0\ge0$ be the minimum integer such that $r_n=r_{n_0}$ for all $n\ge n_0$. Let $0\le m\le n_0$. Then $X_m$ is isomorphic to $p^m$ disjoint copies of $X$.    
\end{proposition}
\begin{proof}
Without loss of generality, we may assume  $\alpha(e)=1$ for all $e\in\EE(X)$ (thanks to Lemma~\ref{lem:iso}). Under the notation of the proof of Proposition~\ref{prop:connected}, let $U_n\leq (\ZZ/p^n\ZZ)$ be the subgroup generated by $\{r_C \mod p^n:C\in \cC\}$.

Recall from \cite[Proposition~2.4]{LM2} that $\ZZ/p^n\ZZ$ acts transitively on the connected components of $X_n$, which are isomorphic to each other. If $a\in\ZZ/p^n\ZZ$ is in the stabilizer of this action, then $(v,0)$ and $(v,a)$ lie in the same connected component of $X_n$. This gives a cycle of {weight $r$ in $\tilde X$} with $r\equiv a\mod p^n$. Therefore, $a\in U_n$.

Conversely, if $a\equiv r_C\mod p^n$ for some $C\in\cC$, there exists $v\in\VV(X)=\VV(\tilde X)$ through which $C$ passes. Then $a$ sends the vertex $(v,0)\in\VV(X)$ to $(v,a)$ under the natural action of $\ZZ/p^n\ZZ$ on $X_n$. {Suppose that $\cC$ consists of the edges $e_1,e_2,\dots, e_{r_C}$ such that the source of $e_1$ and the target of $e_{r_C}$ are $v$. Let $P_i$ be the path consisting of the edges $e_1,\dots ,e_i$ for $i\in\{1,\dots, r_C\}$. Then, the edges $(e_1,0),(e_2,w(P_1)),\dots (e_{r_{C}},w(P_{r_C-1}))$ form a path from $(v,0)$ to $(v,a)$ in $\tilde X_n$}. Therefore, $(v,0)$ and $(v,a)$ lie in the same connected component of $X_n$. In particular, $a\in U_n$. Hence, we deduce that $U_n$ is the stabilizer of the action of $\ZZ/p^n\ZZ$ on the connected components of $X_n$.

It follows from the orbit-stabilizer theorem that 
\[
r_n=[\ZZ/p^n\ZZ:U_n].
\]
Let $n_0\ge0$ be the integer such that
    \[[\Zp:\langle r_C:C\in \cC\rangle]=p^{n_0}.\]
We deduce that $r_n=p^{n_0}$ for $n\ge n_0$ and $r_m=p^m$ for $0\le m\le n_0$. Hence, the proposition follows.    
\end{proof}

\begin{remark}
    Our hypothesis that {$\tilde X$} contains a cycle ensures that $\cC$  is non-empty. When $\cC=\emptyset$, the subgroup $U_n$ in the proof of Proposition~\ref{prop:stabilize} is trivial. In this case, $X_n$ is the disjoint union of $p^n$ copies of $X$, and we do not obtain any $\Zp$-towers from a constant $\Zp$-valued voltage assignment.
\end{remark}

\begin{corollary}
\label{cor:stabilize}
    Under the same notation and assumptions as Proposition~\ref{prop:stabilize}, $(X_n)_{n\ge0}$ consists of $p^{n_0}$ copies of $\Zp$-towers over $X$. 
\end{corollary}
\begin{proof}
Recall that all connected components of $X_n$ are isomorphic to each other. Let $Y$ be a connected component of $X_{n_0}$. For each $n\ge n_0$, let $Y_n$ be the pre-image of $Y$ in $X_n$. Then $(Y_n)_{n\ge n_0}$ is a $\Zp$-tower of $Y$ by \cite[Corollary~4.7]{LM3}.
\end{proof}

\begin{example}
    Take $p=3$ and suppose that $X$ is the directed cycle with 3 vertices. Let $\VV(X)=\{v_i:i\in\ZZ/3\ZZ\}$ and $\EE(X)=\{e_i:i\in\ZZ/3\ZZ\}$ such that $e_i$ goes from $i$ to $i+1$. Suppose that $\alpha(e_i)=1$ for all $i$. Then
    the derived graph $X(\ZZ/3\ZZ,\alpha_{/1})$ consists of three disjoint copies of $X$, induced by $\{(v_i,0),(v_{i+1},1),(v_{i+2},2)\}$, $i=1,2,3$. For $n\ge2$, $X(\ZZ/3^n\ZZ,\alpha_{/n})$ consists of three disjoint copies of directed cycles with $3^{n}$ vertices, given by $\{(v_i,0),(v_{i+1},1),(v_{i+2},2),\dots,(v_i,3^n-3),(v_{i+1},3^n-2),(v_{i+2},3^n-1)\}$.
\end{example}

\begin{defn}\label{def:disconnected}
    Let $X$ be a finite connected directed graph {such that $\tilde X$} contains a cycle. Let $\alpha$ be a constant $\Zp$-valued voltage assignment on $X$ that takes values in $\Zp^\times$. Let $n_0$ be the integer given by Proposition~\ref{prop:stabilize}. Let $\hat{X}_{n_0}$ be a connected component of $X_{n_0}$. For each $n>n_0$, let $\hat X_n$ be the pre-image of $\hat X_{n_0}$ in $X_n$. We shall refer to $(\hat X_n)_{n\ge n_0}$ a \textbf{constant $\Zp$-tower} over $X$. We define $\mu(X)$ and $\lambda(X)$ to be the corresponding invariants given by \eqref{eq:Iw}.
\end{defn}

Remark~\ref{rk:mu-lambda} discusses how to determine $\mu(X)$ and $\lambda(X)$ from $M_X(T)$ when $n_0=0$. We conclude this section with a description of $\mu(X)$ and $\lambda(X)$ in terms of $M_X(T)$ for general $n_0$. The following lemma is a slight generalization of \cite{KM}, where the same result was proved under the assumption that $X_n$ is connected for all $n$.
{Note that the invariants $\mu$ and $\lambda$ as well as the groups $\Pic(\cdot)$, $\Pr(\cdot)$ and $\Div(\cdot)$ have only been defined for undirected graphs. If $X$ is a directed graph, we write $\Pic(X)$ to denote $\Pic(X')$, where $X'$ is the graph obtained from $X$ after applying the forgetful map (and likewise for $\Pr(X)$ and $\Div(X)$).  We write $X_\infty$ for $X(\Zp,\alpha)$, where $\alpha$ is the constant voltage assignment with parameter $1$.}
\begin{lemma}
\label{lemma:projection-pic}
    Let $X$ and $n_0$ be as in Proposition \ref{prop:stabilize}. For all integers $n\ge n_0$, there is an isomorphism
    \[\pi_n\colon \textup{Pic}(X_\infty)\otimes \Z_p[[T]]/((1+T)^{p^n}-1)\to \textup{Pic}(X_n)\otimes \Z_p.\]
\end{lemma}
\begin{proof} Let $\textup{Pr}_{\Z_p[[T]]}(X_\infty)$ be the image of $\textup{Pr}(X_\infty)\otimes \Z_p[[T]]$ in $\textup{Div}(X_\infty)\otimes\Z_p[[T]]$. By definition, there is an isomorphism
\begin{align}\label{tensor-lambda}\textup{Pic}(X_\infty)\otimes \Z_p[[T]]\cong (\textup{Div}(X_\infty)\otimes \Z_p[[T]])/\textup{Pr}_{\Z_p[[T]]}(X_\infty).\end{align}
    There are natural projections
    \begin{align*}
        \pi_n^1&\colon \textup{Pr}_{\Z_p[[T]]}(X_\infty)\to \textup{Pr}(X_n)\otimes \Z_p\\
        \pi_n^2&\colon \textup{Div}(X_\infty)\otimes \Z_p[[T]]\to \textup{Div}(X_n)\otimes \Z_p.
    \end{align*}
    The kernel of $\pi_n^1$ is given by $((T+1)^{p^n}-1)(\textup{Div}(X_\infty)\otimes \Z_p[[T]])$. The snake lemma applied to the commutative diagram
     \[\begin{tikzcd}
0\arrow[r]&\textup{Pr}_{\Z_p[[T]]}(X_\infty)\arrow[r]\arrow[d,"\pi_n^1"]&\textup{Div}(X_\infty)\otimes\Z_p[[T]]\arrow[r]\arrow[d,"\pi_n^2"]&\textup{Pic}(X_\infty)\otimes \Z_p[[T]]\arrow[r]\arrow[d,"\pi_n^3"]&0\\
  0\arrow[r]&\textup{Pr}(X_n)\otimes \Z_p\arrow[r]&\textup{Div}(X_n)\otimes \Z_p\arrow[r]&\textup{Pic}(X_n)\otimes \Z_p\arrow[r]&0
  \end{tikzcd}\]  
    implies that 
    \[\pi_n^3\colon \textup{Pic}(X_\infty)\otimes \Z_p[[T]]\to \textup{Pic}(X_n)\]
    is surjective and that \[\ker(\pi_n^3)=((T+1)^{p^n}-1)(\textup{Pic}(X_\infty)\otimes \Z_p[[T]]),\]
    which results in the isomorphism $\pi_n$ asserted in the statement of the lemma.
\end{proof}
\begin{proposition}\label{prop:disconnected}
    Let $X$ and $n_0$ be as in Proposition~\ref{prop:stabilize}. Then
    \[
    M_X(T)=\left(p^{\mu(X)}Tg(T)\right)^{p^{n_0}}u(T),
    \]
    where $g(T)$ is a distinguished polynomial of degree $\lambda(X)$ and $u(T)\in\Zp[[T]]^\times$.
\end{proposition}
\begin{proof}
Let $k$ be the number of vertices in $X$. Then $$\Pic(X_\infty)\otimes\Zp[[T]]\cong\frac{\Zp[[T]]^{\oplus k}}{(D-A(1+T)-A^t(1+T)^{-1}){\Z_p[[T]]^{\oplus k}}}$$
by \eqref{tensor-lambda}, combined with the description of the principal divisors given in \cite[\S4]{KM1} \footnote{Note that the discussion in loc. cit. imposes the assumption that $X_n$ is connected for all $n$. Nonetheless, the description of the principal divisors therein is still valid without this assumption.}. Furthermore,
    Lemma \ref{lemma:projection-pic} says that
    \[
\frac{    \Pic(X_\infty)\otimes\Zp[[T]]}{((1+T)^{p^n}-1){\Pic(X_\infty)\otimes\Zp[[T]]}}\cong \Pic(X_n)\otimes \Z_p
    \]
    for all $n\ge n_0$. Hence, we deduce that
        \begin{align*}
     (\Pic(\hat X_n)\otimes\Zp)^{p^{n_0}}
    & \cong\Pic(X_n)\otimes\Zp\\&\cong\frac{\Lambda^{\oplus k}}{{(D-A(1+T)-A^t(1+T)^{-1})\Lambda^{\oplus k}+((1+T)^{p^n}-1)\Lambda^{\oplus k}}}.
        \end{align*}
      
    As $(\hat X_n)_{n\ge n_0}$ is a $\Zp$-tower,
\[\left(\varprojlim_{n\ge n_0}\Pic(\hat X_n)\otimes\Zp\right)^{p^{n_0}}\cong
\left(\frac{\Lambda^{\oplus k}}{F(T){\Lambda^{\oplus k}}}\right)^{p^{n_0}}\cong\frac{\Lambda^{\oplus k}}{(D-A(1+T)-A^t(1+T)^{-1})\Lambda^{\oplus k}},\]
    where $F(T)$ is a $k\times k$ matrix given by some $\Zp$-voltage assignment on $X$ as in Remark~\ref{rk:undirected}. Furthermore, $\det F(T)=p^{\mu(X)}Tg(T)v(T)$, where $g(T)$ is a distinguished polynomial of degree $\lambda(X)$ and $v(T)\in\Zp[[T]]^\times$. The isomorphism above implies that the power series $\det F(T)^{p^{n_0}}$ and $M_X(T)$ agree up to a unit of $\Zp[[T]]$. Hence, the proposition follows.
    \end{proof}

\section{Results on the $\mu$-invariant of a constant $\Zp$-tower}
We prove Theorem~\ref{thmB} in this section. We then apply it to study the  $\mu$-invariant of towers of isogeny graphs introduced in \cite{LM2}.

\subsection{Proof of Theorem~\ref{thmB}}
We begin with part (a) of the theorem.
\begin{theorem}\label{thm:mu>0}
    Let $X$ be a finite connected directed graph {such that $\tilde X$} contains a cycle. Let $(X_n)_{n\ge0}$ be a constant $\Zp$-tower over $X$. Suppose that $p$ divides the in-degree and out-degree of each vertex of $X$, then the $\mu$-invariant of the tower $(X_n)_{n\ge0}$ is positive.
\end{theorem}
\begin{proof}
  Consider
   \[
M_X(T)=\det\left(D-A(1+T)-A^t(1+T)^{-1}\right),
   \]
   as in Definition~\ref{def:MXT}. The sum of the columns of the matrix $D-A(1+T)-A^t(1+T)^{-1}$ is a column vector whose entries are given by
   \[
   d_i(v)+d_o(v)-d_o(v)(1+T)-d_i(v)(1+T)^{-1}, v\in\VV(X).
   \]
   Therefore, if $p|d_i(v)$ and $p|d_o(v)$ for all $v\in\VV(X)$, the power series $M_X(T)$ is divisible by $p$. Hence,  $\mu(X)>0$ by Proposition~\ref{prop:disconnected}.
\end{proof}

Note that if $X$ is the $l$-isogeny graph of supersingular elliptic curve studied in \cite{LM2}, the in-degree and out-degree of each vertex are both equal to $l+1$. In particular, if $p|(l+1)$, we would obtain a $\Zp$-tower that has positive $\mu$-invariant. This is a special case of Example~5.7 discussed in \textit{op. cit.}

We now prove part (b) of Theorem~\ref{thmB}, which can be considered as a partial converse of part (a).

\begin{theorem}
\label{thm:mu=0}
    Let $X$ be a finite connected directed graph {such that $\tilde X$} contains a cycle. Assume that $d_i(v)+d_o(v)$ takes the same value $k$ for all $v\in\VV(X)$. Furthermore, assume that $p\nmid k$ and that the adjacency matrix $A$ is normal. Then $\mu(X)=0$. 
\end{theorem}
\begin{proof} By Proposition~\ref{prop:disconnected}, it suffices to show that the power series
    \[
M_X(T)=\det\left(D-A(1+T)-A^t(1+T)^{-1}\right)
   \]
   is not divisible by $p$.
The valency matrix $D$ is equal to $kI$. Since $A$ is a normal matrix, there exists a unitary matrix $U$ such that $U^*AU=\mathrm{diag}(\alpha_1,\dots,\alpha_r)$ for some complex numbers $\alpha_1,\dots,\alpha_r$. Note that since $A$ is defined over $\ZZ$, all the eigenvalues $\alpha_i$ are algebraic integers. Furthermore, we have $U^*A^tU=\mathrm{diag}(\overline\alpha_1,\dots,\overline\alpha_r)$, where $\overline\alpha_i$ denotes the complex conjugation of $\alpha_i$.

   We have
   \begin{align*}
   M_X(T)&=\prod_{i=1}^r\left(k-\alpha_i(1+T)-\overline\alpha_i(1+T)^{-1}\right)\\
   &=\prod_{i=1}^r\left(k-\alpha_i-\overline\alpha_i+(\overline{\alpha}_i-\alpha_i)T+\overline\alpha_i(-T^2+T^3-\dots)\right).
        \end{align*}
   Therefore, it suffices to prove that each power series \[k-\alpha_i-\overline\alpha_i+(\overline{\alpha}_i-\alpha_i)T-\overline\alpha_i(-T^2+T^3-\dots)\] does not belong to $\fm[[T]]$, where $\fm$ denotes the maximal ideal of the {ring of integers of the }field $\Qp(\alpha_1,\dots,\alpha_r)$.

   If $v_p(k-\alpha_i-\overline\alpha_i)=0$, then our claim is immediately valid. If there exists $i$ such that $v_p(k-\alpha_i-\overline\alpha_i)>0$, then $v_p(\alpha_i+\overline\alpha_i)=0$ since $p\nmid k$. As $$\alpha_i+\overline\alpha_i=2\overline\alpha_i-(\overline\alpha_i-\alpha_i),$$
   we have either $v_p(\overline{\alpha_i})=0$ or $v_p(\overline\alpha_i-\alpha_i)=0$.  Thus, our claim follows.
\end{proof}

\begin{remark}
If $A$ is symmetric, which is equivalent to that every edge in $X$ admits an edge going in the opposite direction, then $A$ is normal. 

More generally, as discussed in \cite[P.2]{Jorgensen}, the adjacency matrix $A$ is normal if and only if for any two vertices (not necessarily distinct) they have the same number of common out-neighbors as common in-neighbors.
\end{remark}

\subsection{Application to supersingular isogeny graphs}\label{S:isogeny}

In this subsection,  we apply the general results obtained so far to the isogeny graphs introduced in \cite{LM2}. Throughout, $p$ and $r$ are fixed distinct prime numbers, $N$ is a natural number coprime to $rp$, and $\ell$ is a prime such that $\ell\equiv 1\pmod{Np}$. Recall that $Y^{\ss}$ is the supersingular subgraph of $Y_l^q(Np)$, the vertices of which are isomorphic classes of elliptic curves defined over $\FF_q$, where $q$ is an even power of $r$, together with a level structure $(R_1,R_2,\zeta)$ such that $\{R_1,R_2\}$ is a basis of $E[N]$ and $\zeta$ is a primitive $p^n$-th root of unity (see Definition \ref{def:intro}). For $n\ge1$, $Y^{\ss}_n$ is the pre-image of $Y^\ss$ in $Y_l^q(Np^n)$ . 

\begin{lemma}\label{lem:Y0}
Let $Y_0$ be the image of $Y^\ss$ in $Y_l^q(N)$ under the natural projection map $Y_l^q(Np)\rightarrow Y_l^q(N)$. Let $\beta_n$ be the $(\ZZ/p^{n}\ZZ)^\times$-valued voltage assignment on $Y_0$ that sends every edge to $l\mod p^n$. Then $$Y_n^\ss\cong Y_0((\ZZ/p^{n}\ZZ)^\times,\beta_n).$$
   \end{lemma}
\begin{proof}
Let $S$ be the set of representatives of elliptic curves given as in the introduction. For each $E\in S$, we fix a basis $\{s_E,t_E\}$ of the Tate module $T_p(E)$ so that the Weil pairing $\langle s_E,t_E\rangle_E$ is a topological generator of $\Zp(1)$ that is independent of $E$. 

Let $e\in \EE(Y_0)$ be an edge that arises from an $l$-isogeny $\phi:E\to E'$, where $E,E'\in S$. We define $g_e\in\GL_2(\Zp)$ to be the matrix
\[
\begin{pmatrix}
    \phi(s_E)\\ \phi(t_E)
\end{pmatrix}=g_e\cdot \begin{pmatrix}
    s_{E'}\\t_{E'}
\end{pmatrix},
\]
as in \cite[Definition~2.7]{LM2}. If we write $\beta_n$ to be the voltage assignment on $Y_0$ that sends $e$ to $\det(g_e)\mod p^n$, Proposition 5.2 of \textit{op. cit.} says that
$$Y_n^\ss\cong Y_0((\ZZ/p^{n}\ZZ)^\times,\beta_n).$$
As discussed in Remark 5.3 of \textit{op. cit.}, since $\phi$ is an $l$-isogeny, we have $$\langle \phi (s_E),\phi(t_E)\rangle_{E'}=\langle s_E,t_E\rangle_E^l.$$
Thus, $\det(g_e)=l$, which implies that $\beta_n$ sends every edge of $Y_0$ to $l\mod p^n$, as desired.
\end{proof}

\begin{lemma}\label{lem:constant}
Suppose that $p>2$. Let $\alpha$ be a constant $\Zp$-voltage assignment on $Y^\ss$ that takes values in $\Zp^\times$.  Then $Y_{n+1}^\ss\cong Y^\ss(\ZZ/p^n\ZZ,\alpha_{/n})$ for all $n\ge0$.
\end{lemma}
\begin{proof}
Let $Y_0$ and $\beta_n$ be defined as in the statement of Lemma~\ref{lem:Y0}.
As we have assumed that $l\equiv 1\mod p$, the graph $Y^\ss$ is isomorphic to $p-1$ disjoint copies of $Y_0$ and the image of $\beta_n$ lies inside the subgroup $U_n^{(1)}$ of $(\ZZ/p^n\ZZ)^\times$ that consists of elements that are congruent to $1$ modulo $p$. Consequently, we deduce that
\[
Y_n^\ss\cong Y_0((\ZZ/p^{n}\ZZ)^\times,\beta_n)\cong Y^\ss(U_n^{(1)},\beta_n)
\]
for all $n\ge1$.
Therefore, the lemma follows after taking $\alpha$ to be the composition
\[
\EE(Y^\ss)\rightarrow 1+p\Zp\cong\Zp,
\]
where the first arrow sends all edges to $l$ (which belongs to $1+p\Zp$ as $l\equiv 1\mod p$) and the last isomorphism is given by $\frac{1}{p}\cdot \log_p$. Here, $\log_p$ denotes the $p$-adic logarithm.
\end{proof}

\begin{proposition}\label{prop:constant}
Assume that $p>2$. Then  $m_0$ is the minimal index where the number of connected components of $Y_n^\ss$ stabilizes. Furthermore, the number of connected components of $Y_n^\ss$ and $X_n^\ss$ are the same. 
\end{proposition}
\begin{proof}
    By \cite[Theorem 4.3]{LM2}, the number of connected components of $X_n^\ss$ is given by $[(\ZZ/p^nN\ZZ)^\times :\langle \ell\rangle]$. As we assume that $\ell\equiv 1\pmod{Np}$, this number is equal to $\varphi(N)[(\ZZ/p^n\ZZ)^\times:\langle \ell \rangle]$, where $\varphi$ is Euler's totient function.  Let $v=(E,R_1,R_2,\zeta)$ be a vertex of $Y^\ss_n$. Let $\phi$ be an $\ell$-isogeny on $E$. We have
    \[\widehat{\phi}\circ \phi(E,R_1,R_2,\zeta)=(E,R_1,R_2,\zeta^{\ell^2}),\] 
    meaning that $(E,R_1,R_2,\zeta)$ and $(E,R_1,R_2,\zeta^{\ell^2})$ lie in the same connected component of $Y_n^\ss$.
    Note that $Y_0$ has $\varphi(N)$ connected components.
    Thus, the number of connected components of $Y_n^{\ss}$ is bounded above by $$[(\ZZ/p^n\ZZ)^\times:\langle \ell^2 \rangle]\varphi(N)=[(\ZZ/p^n\ZZ)^\times:\langle \ell \rangle]\varphi(N),$$ 
    as $p$ is odd.
    
    If there is a path from $(E,R_1,R_2,\zeta)$ to $(E,R_1,R_2,\zeta^k)$ in $Y_n^\ss$, the integer $k$ has to be a power of $\ell$. Therefore, the number of connected components of $Y_n^\ss$ is at least $\varphi(N)[(\ZZ/p^n\ZZ)^\times:\langle \ell \rangle]$.
    This shows that indeed $m_0$ is the minimal index where the number of connected components of $Y_n^\ss$ stabilizes as $n$ increases and that this number is equal to the number of connected components of $X_n^\ss$.  
\end{proof}

For each $n\ge m_0$, let $Z_n'$ be the image of $Z_n$ in $Y_n^{\ss}$ under the natural projection $X_n^\ss\rightarrow Y_n^\ss$. Combining Lemma~\ref{lem:constant} with Propositions~\ref{prop:disconnected} and \ref{prop:constant} implies that $(Z'_n)_{n\ge m_0}$ is a constant $\Z_p$-tower over $Z_{m_0}'$.  Note that $Z_n$ is a covering of $Z'_{m_0}$ with Galois group
\[G_n=\left\{x\in \textup{GL}_2(\ZZ/p^nN\ZZ)\mid \det(x)\equiv 1 \pmod {p^{m_0}N},\ x\equiv 1\pmod {Np}\right\}.\] Let $G=\varprojlim_n G_n$. Then $(Z_n)_{n\ge m_0}$ can be regarded as a "$G$-tower" of $Z'_{m_0}$ and $(Z'_n)_{n\ge m_0}$ is a "$\Z_p$-subtower" (in the language of \cite{KM}).

\begin{corollary}\label{cor:isogeny} Assume that $p>2$.
    Then $\mu(Z'_{m_0})=0$. 
\end{corollary}
\begin{proof}
As $\ell\equiv 1\pmod{Np}$, an isogeny and its dual generate edges in opposite directions in $X^\ss$. Therefore, the adjacency matrix of $X$ is symmetric and in particular normal. Furthermore, the in-degree and out-degree of each vertex are both equal to $l+1$. As $l\equiv 1\mod p$, we have $p\nmid (l+1)$. Combining Lemma~\ref{lem:constant} and Theorem \ref{thm:mu=0} gives $\mu(Z_{m_0}')=0$. 
\end{proof}

\begin{remark}
    We studied $\Zp$-towers of isogeny graphs in \cite{LM1,LM3}, where $p=r$. The results in the present article do not apply immediately to these towers since they are not a priori constant $\Zp$-towers.
\end{remark}

Let $H\subset G$ be the {maximal} subgroup of $G$ that acts trivially on $Z'_n$ for all $n\ge m_0$. Let $\Lambda=\Z_p[[G]]$. 
 We say that the $\mathfrak{M}_H(G)$-property holds for $(Z_n)_{n\ge m_0}$ if 
\[(\textup{Pic}(Z_\infty)\otimes \Lambda)/(\textup{Pic}(Z_\infty)\otimes \Lambda)[p^\infty]\]
is finitely generated as a $\Z_p[[H]]$-module. This formulation is the analogue of the $\mathfrak{M}_H(G)$-property in classical non-commutative Iwasawa theory introduced in \cite{CFKSV}. 
If the $\mathfrak{M}_H(G)$-property holds for the tower $(Z_n)_{n\ge m_0}$, one can formulate a $K$-theoretic Iwasawa main conjecture. To do so, let $S$ be the subset of all elements in $\Lambda$ such that $\Lambda /(f)$ is finitely generated over $\Z_p[[H]]$. Define $S^*=\bigcup_{k\ge0}p^kS$. Let $K_1(\Lambda_{S^*})$ be the $K_1$ group of ${S^*}^{-1}\Lambda$ and let $K_0(\Lambda \textup{ on } S^*)$ be the $K_0$ group of the category of bounded complexes of projective $\Lambda$-modules that are ${S^*}^{-1}\Lambda$ exact. Then there is a natural differential
\[\partial \colon K_1(\Lambda_{S^*})\to  K_0(\Lambda \textup { on }S^*).\]
We recall the following theorem.
\begin{theorem}\label{main-conj}\cite[Theorem 7.3]{KM} If the $\mathfrak{M}_H(G)$-property holds for the tower $(Z_n)_{n\ge m_0}$, 
    there exists an element $P\in  K_1(\Lambda_{S^*})$ interpolating the Ihara L-function $h(\rho,Z_n/Z_{m_0},1)$ such that $\partial(P)=[\Pic(Z_\infty)\otimes \Lambda)]\in K_0(\Lambda \textup{ on } S^*)$.
\end{theorem}
\begin{corollary} Assume that $p>2$. 
    The $\mathfrak{M}_H(G)$-property holds for the tower $(Z_n)_{n\ge m_0}$. In particular, the conclusion of Theorem \ref{main-conj} holds for this tower.
\end{corollary}
\begin{proof} By Corollary \ref{cor:isogeny} $\mu(Z'_{m_0})=0$. By \cite[Corollary 8.2]{KM} the vanishing of $\mu$ implies that the $\mathfrak{M}_H(G)$-property is valid. \end{proof}

\section{Balanced direct graphs}

We prove Theorem~\ref{thmD} in this section. We first recall the following well-known result on the number of spanning trees in an undirected graph.
\begin{theorem}[Kirchoff's formula]\label{thm:Kirchoff}
    Let $X$ be a finite connected directed graph. Let $L=D-A-A^t$ denote the Laplacian matrix of the undirected graph $X'$ given by the image of $X$ under the forgetful map. Then the number of spanning trees in $X'$ is equal to $(-1)^{i+j}\det(L_{ij})$, where $L_{ij}$ is the minor of $L$ after deleting the $i$-th row and the $j$-th column of $L$.
\end{theorem}

\begin{theorem}\label{thm:regular-balanced}
      Let $X$ be a finite connected directed graph that is balanced (i.e. $d_i(v)=d_o(v)$ for all $v\in\VV(X)$). Suppose that $p>2$ or {$\tilde X$} contains a cycle whose weight is coprime to $p$. Furthermore, assume that $p\nmid k\kappa_X$, where $k$ is the total degree $\sum_{v\in\VV(X)}d_i(v)$.  Then the $\mu$-invariant of a constant $\Zp$-tower of $X$ is zero, whereas the $\lambda$-invariant is one.
\end{theorem}
\begin{proof}
    Let us label the vertices $v_1,\dots, v_r$ and write $d_m=d_i(v_m)=d_o(v_m)$ for $m=1,\dots, r$.
    Consider the matrix
        \[
D-A(1+T)-A^t(1+T)^{-1}.
   \]
If we add every column after the first one to the first, we get
\begin{align*}
M_X(T)&=\det\begin{bmatrix}
    2d_1-d_1(1+T)-d_1(1+T)^{-1}& R_1'\\
    2d_2-d_2(1+T)-d_2(1+T)^{-1}&R_2'\\
    \vdots&\vdots\\
    2d_r-d_r(1+T)-d_k(1+T)^{-1}&R_r'
\end{bmatrix}\\
&=\left(2-(1+T)-(1+T)^{-1}\right)\det\begin{bmatrix}
    d_1& R_1'\\
    d_2&R_2'\\
    \vdots&\vdots\\
    d_r&R_r'
\end{bmatrix}\\
&=(-T^2+T^3-\dots)\left(d_1\det\begin{bmatrix}
    R_2'\\R_3'\\ \vdots\\ R_r'
\end{bmatrix}-d_2\det\begin{bmatrix}
    R_1'\\R_3'\\ \vdots\\ R_r'
\end{bmatrix}+\cdots\right)
    \end{align*}
where $R_i'$ denotes the $i$-th row of $D-A(1+T)-A^t(1+T)^{-1}$ after removing the first entry. Then
\[
\left.\det\begin{bmatrix}
    R_2'\\ R_3'\\
    \vdots\\
    R_r'
\end{bmatrix}\right|_{T=0}=\det(L_{11})=\kappa_X,\left.\quad\det\begin{bmatrix}
    R_1'\\ R_3'\\
    \vdots\\
    R_r'
\end{bmatrix}\right|_{T=0}=\det(L_{21})=-\kappa_X,\dots
\]
by Theorem~\ref{thm:Kirchoff}. Therefore, 
\[
\left.\left(d_1\det\begin{bmatrix}
    R_2'\\R_3'\\ \vdots\\ R_r'
\end{bmatrix}-d_2\det\begin{bmatrix}
    R_1'\\R_3'\\ \vdots\\ R_r'
\end{bmatrix}+\cdots\right)\right|_{T=0}=(d_1+d_2+\dots)\kappa_X=k\kappa_X.
\]
 We deduce that
\[
M_X(T)\in-k\kappa_X T^2+T^3\Zp[[T]].
\]
Thus, under the assumption that $p\nmid k\kappa_X$, the $\mu$-invariant and $\lambda$-invariant of the power series $M_X(T)$ are $0$ and $2$, respectively. If $p=2$, as we have assumed that $\tilde X$ contains a cycle of weight coprime to $p$, the constant $n_0$ given by Proposition~\ref{prop:disconnected} is equal to $0$. Otherwise, if $p>2$
\[
\left(p^{\mu(X)}Tg(X)\right)^{p^{n_0}}u(X)\in-k\kappa_X T^2+T^3\Zp[[T]]
\]
forces $n_0$ to be zero.
In both cases, we have $\mu(X)=0$ and $\lambda(X)=1$.
\end{proof}

\begin{remark}\label{rk:lambda1}
Assume that $X$ is a finite connected balanced graph. If $p=2$ and there is no cycle of weight coprime to $p$, then $\mu(X)=0$ and $0\le \lambda(X)\le 1$. If $p>2$ then $\lambda(X)\ge 1$ as $T^2$ divides $M_X(T)$.
\end{remark}

    In the case of the isogeny graphs considered in \S\ref{S:isogeny}, the base graph is balanced and regular. Thus, we can apply Theorem~\ref{thm:regular-balanced} to deduce the following corollary.
    \begin{corollary}\label{cor:isogeny-regular}
Let $Z_{m_0}'$ be the graph defined as in \S\ref{S:isogeny}.   If $p\nmid 2r\kappa_{{Z'_{m_0}}}$, where $r$ denotes the number of vertices in $Z_{m_0}'$, then $\mu(Z_{m_0}')=0$ and $\lambda(Z_{m_0}')=1$.
    \end{corollary}

\section{Constant $\Zp$-towers over volcano graphs}
\label{S:volcano}

Our goal is to study properties of constant $\Zp$-towers over a volcano graph. In particular, we prove Theorem~\ref{thmE}. We first recall the definition of a volcano graph.

\begin{defn}
A finite undirected graph is said to be an \textbf{abstract crater} if it is one of the following graphs:
\begin{itemize}
    \item A single vertex with two self-loops;
    \item A cycle graph (which can be a single vertex together with a self-loop, or a graph consisting of two vertices with two edges between them).
    \item A single vertex without any loops.
\end{itemize}
    
    A finite undirected graph $X$ is said to be an \textbf{abstract $l$-volcano} of depth $d\ge0$ if $\displaystyle\VV(X)=\bigsqcup_{i=0}^dV_i$ such that there exists a positive integer $l$ such that:
    \begin{itemize}
        \item The subgraph induced by $V_0$ is an abstract crater;
        \item If, $d\ge 1$, all vertices in $V_0\bigsqcup \dots \bigsqcup V_{d-1}$ have degree\footnote{To compute the degree of a vertex we only count loops with multiplicity one, not with multiplicity $2$. } $l+1$, whereas those in $V_d$ have degree $1$; 
        \item If $u\in V_r$ and $v\in V_k$ are connected by an edge, then $|r-k|\le 1$;
        \item For all $0<r\le d$, the subgraph induced by $V_r$ is totally disconnected;
        \item For $0<r\le d$, each vertex in $V_r$ admits exactly one edge to a vertex in $V_{r-1}$.
    \end{itemize}
    The subgraph induced by $V_0$ is called the \textbf{crater} of $X$.
The vertices in $V_r$ are said to be at \textbf{level} $r$. We call the maximal level of $X$ the \textbf{depth} of $X$. 

    \end{defn}
These graphs are utilized to describe the isogeny graphs of ordinary elliptic curves in \cite{kohel}, where $l$ is a prime number; we refer the reader to \cite[\S4]{pazuki} for further details. For our purposes, it is not necessary to assume that $l$ is a prime number.

\begin{defn}
Let $X$ be an abstract $l$-volcano. We define a directed graph $\tilde X$ by assigning an orientation to the edges of $X$ as follows: 
    \begin{itemize}
        \item Each edge between $v$ of level $i$ and $w$ of level $i+1$ becomes an oriented edge with origin $v$ and target $w$. 
        \item Let $k$ be the number of vertices in the crater of $X$. Enumerate the vertices of the crater as $\{v_i:i\in
        \ZZ/k\ZZ\}$ such that $v_i$ and $v_j$ are adjacent if and only if $|i-j|\equiv 1\mod k$. We assign the oriented edges as the ones with origin $v_i$ and target $v_{i+1}$. 
    \end{itemize}
   We shall refer to a constant voltage assignment on $\tilde X$  as a constant voltage assignment on $X$. Similarly, we shall call the image of a constant $\Zp$-tower over $\tilde X$ under the forgetful map a constant $\Zp$-tower over $X$.\end{defn}

   \begin{remark}Let $X$ be an abstract $l$-volcano whose crater is a cycle graph with at least two vertices.
   As an undirected graph, $X$ has as many edges as vertices, which implies that $H_1(X,\Z)\cong\Z$. In particular, there is (up to isomorphism) only one $\Z_p$-tower over $X$. If the crater consists only of one vertex without any loops, $X$ is a tree.  In this case, there is no $\Zp$-tower over $X$.
   \end{remark}

\begin{lemma}\label{lem:connected-volcano}
Let $X$ be an abstract $l$-volcano.  Let $\alpha$ be a constant $\Zp$-valued voltage assignment on $X$. Then $X_n=X(\ZZ/p^n\ZZ,\alpha_{/n})$ is connected for all $n$ if and only if the image of $\alpha $ is in $\Zp^\times$ and the crater of $X$ is not a single vertex without any edges, and not a cycle graph whose length is divisible by $p$. If the crater is a cycle of length $p^{n_0}a$, where $p\nmid a$, then $(X_n)_{n\ge 0}$ is the disjoint union of $p^{n_0}$ copies of a $\Zp$-tower over $X$.
\end{lemma}
\begin{proof}
    This follows from Corollary~\ref{cor:exist} and Proposition~\ref{prop:disconnected}.
\end{proof}

Note that a constant $\Zp$-tower over an abstract $l$-volcano exists if and only if the crater of $X$ contains at least one edge.

\begin{defn}
      Let $\alpha$ be a constant $\Zp$-valued voltage assignment on $X$. We say that a vertex in the derived graph $X_n:=X(\ZZ/p^n\ZZ,\alpha_{/n})$ is at level $r$ if it is of the form $(v,a)\in\VV(X)\times\ZZ/p^n\ZZ$, where $v$ is a vertex at level $r$. We write $\VV_r(X_n)$ for the set of level-$r$ vertices in $X_n$.
\end{defn}

\begin{proposition}
\label{coverings of undirected volcanos}
    Let $X$ be an abstract $l$-volcano. Assume that the crater is a cycle of length $p^{n_0}a$ where $p\nmid a$ (i.e. an abstract crater that contains a single vertex, without edges or with two self-loops is excluded). If $(\hat X_n)_{n\ge n_0}$ is a constant $\Zp$-tower over $X$, then $\hat X_n$ is an abstract $l$-volcano of the same depth as $X$ for all $n\ge n_0$. Furthermore, the crater of each $\hat X_n$ is a cycle for all $n\ge n_0$.
\end{proposition}
\begin{proof}
Let $\alpha$ be a constant assignment cover on $X$. Let $d$ be the depth of $X$ and $n\ge0$ an integer. Let $v\in\VV(X_n)$. It follows from the definition of $X(\ZZ/p^n\ZZ,\alpha_{/n})$ that if $v$ is of level $d$, then it is connected to exactly one vertex in $\VV_{d-1}(X_n)$. If it is of level $r$ with $1\le r\le d-1$, then it is connected to exactly one vertex in $\VV_{r-1}(X_n)$ and $l$ vertices in $\VV_{r+1}(X_n)$. If it is in the crater, then the number of level-1 vertices connected to $v$ is constant, equal to its counterpart in $X$.

It remains to show that $\VV_0(X_n)$ consists of $p^{n_0}$ copies of a cycle, which is an abstract crater. Let $\cX_n$ denote the subgraph of $X_n$ induced by $\VV_0(X_n)$. If $\beta_n$ denotes the restriction of $\alpha_{/n}$ to $\cX_0$, we have $\cX_n=X(\ZZ/p^n\ZZ,\beta_n)$.

 Let $\{v_i:i\in\ZZ/p^{n_0}a\ZZ\}$ be the vertices of $\VV_0(X_0)$ and assume that there is an edge going from $v_i$ to $v_{j}$ whenever $j\equiv i+1\mod p^{n_0}a$. 
 The edges of $\VV_0(X_n)$ are of the form $(v_i,a)\rightarrow (v_{i+1},a+1)$. In particular, there is a path of length $j$ from $(v_i,a)$ to $(v_{i+j},a+j)$. This path becomes a cycle if and only if $j\equiv 0\mod p^{n_0}a$ and $j\equiv 0\mod p^n$. As $a$ is coprime to $p$ by Lemma~\ref{lem:connected-volcano}, this is equivalent to $j\equiv 0\mod p^{n}a$. Therefore, $\VV_0(X_n)$ consists of $p^{n_0}$ copies of a cycle of length $p^na$, as desired.
\end{proof}

We are now ready to prove Theorem~\ref{thmE}(a):
\begin{corollary}\label{cor:volcano}
    Let $X$ be an abstract $l$-volcano whose crater is a cycle.  Then $\mu(X)=0$ and $\lambda(X)=1$.
\end{corollary}
\begin{proof}
    The proof of Proposition \ref{coverings of undirected volcanos} tells us that $\hat X_{n}$ is an  abstract $l$-volcano whose crater is a cycle of length $p^na$, where $p\nmid a$. The spanning trees for such a graph are obtained from deleting exactly one of the edges in the crater. In particular, the number of spanning trees is exactly $p^na$. Therefore, $\mu(X)=0$ and $\lambda(X)=1$. 
\end{proof}

We introduce the following definition to study abstract volcano graphs in which the crater consists of a single vertex with two self-loops, which were excluded in Proposition~\ref{coverings of undirected volcanos}.

\begin{defn}\label{def:double}
    A finite undirected graph is said to be a \textbf{double crater} if its set of vertices is of the form $\{v_{1},\dots,v_s\}$ for some positive integer $s\ge2$ such that there are exactly two edges between $v_i$ and $v_j$ whenever $|i-j|\equiv 1\mod s$ and there is no edge between $v_i$ and $v_j$ otherwise.
    
    A finite undirected graph $X$ is said to be an \textbf{abstract augmented $l$-volcano} of depth $d\ge0$ if $\displaystyle\VV(X)=\bigsqcup_{i=0}^dV_i$ such that there exists a positive integer $l$ such that:
    \begin{itemize}
        \item The subgraph induced by $V_0$ is an abstract double crater;
        \item All vertices in $V_1\bigsqcup \dots \bigsqcup V_{d-1}$ have degree $l+1$, whereas those in $V_0$ have degree $l+3$ and those in $V_d$ have degree $1$. If $d=0$ all vertices have degree $4$.
        \item If $u\in V_r$ and $v\in V_k$ are connected by an edge, then $|r-k|\le 1$;
        \item For all $0<r\le d$, the subgraph induced by $V_r$ is totally disconnected;
        \item For $0<r\le d$, each vertex in $V_r$ admits exactly one edge to a vertex in $V_{r-1}$.
    \end{itemize}
    As before, the subgraph induced by $V_0$ is called the \textbf{crater} of $X$.
The vertices in $V_r$ are said to be at \textbf{level} $r$.
\end{defn}

\begin{lemma}\label{lem:derived-double}
    Let $X$ be a depth-zero abstract $l$-volcano whose crater consists of one vertex together with two self-loops. Let $\alpha$ be a constant $\Zp$-valued voltage assignment on $X$. If the image of $\alpha$ lies in $\Zp^\times$, then $X(\ZZ/p^n\ZZ,\alpha_{/n})$ is a double crater with $p^n$ vertices.
\end{lemma}
\begin{proof}
   Let $v$ be the unique vertex of $X$. Suppose that the image of $\alpha$ is $\{1\}$. The vertices of  $X(\ZZ/p^n\ZZ,\alpha_{/n})$ are of the form $(v,a)$, where $a\in\ZZ/p^n\ZZ$. Furthermore, there are two edges between $(v,a)$ and $(v,a+1)$, whereas there is no edge between $(v,a)$ and $(v,b)$ when $|a-b|\not\equiv1\mod p^n$. This coincides with the definition of a double crater.
\end{proof}
This allows us to prove Theorem~\ref{thmE}(b):

\begin{corollary}
    Let $X$ be an abstract $l$-volcano whose crater is a single vertex with two loops. Let $\alpha$ be a constant voltage cover on $X$.  Then $X_n:=X(\ZZ/p^n\ZZ,\alpha_{/n})$ is an abstract augmented $l$-volcano for all $n\ge1$. Furthermore, $\lambda(X)=1$ and $\mu(X)=v_p(2)$.
\label{cor:augmented}
\end{corollary}
    
\begin{proof}
    The proof of the first assertion is essentially the same as Proposition \ref{coverings of undirected volcanos}, except the description of $\cX_n=X(\ZZ/p^n\ZZ,\beta_n)$, which is given by Lemma~\ref{lem:derived-double}.

    It remains to determine the Iwasawa invariants. The number of spanning trees of $X_n$ is equal to the number of spanning trees of $\cX_n$. As it is a double crater of length $p^n$,  the number of spanning trees is $2^{p^n-1}p^n$. Thus, $\mu(X)=v_p(2)$ and $\lambda(X)=1$.
\end{proof}

\begin{remark}
    We emphasize that the constant $\Zp$-towers studied by Theorem~\ref{thmE} are distinct from the ones constructed in \cite{LM1,LM2}. Indeed, the $\Zp$-towers in the aforementioned work consist of the so-called tectonic volcanoes, where the craters are not cycles. Indeed, the in-degree and out-degree of each vertex are 2, not 1. 

    In \cite{kohel,pazuki}, the volcanic structure of the isogeny graphs of ordinary elliptic curves is obtained after identifying an isogeny with its dual. More precisely, suppose $E$ and $E'$ give rise to two vertices in the graph $Y_l^q(1)$. If $\phi:E\rightarrow E'$ is an $l$-isogeny, there is a dual isogeny $\hat\phi:E'\rightarrow E$ of the same degree. These two isogenies result in two directed edges going in opposite directions in $Y_l^q(1)$. These two edges are identified as one single undirected edge in \cite{kohel,pazuki}.
\end{remark}

We conclude the article with a brief discussion on constant $\Zp$-towers that arise from isogeny graphs of ordinary elliptic curves that are, as discussed above, closely related to abstract $l$-volcanoes.
\begin{lemma}
\label{tot:degree}
    Let $X$ be an abstract $l$-volcano of depth $d$. Assume that the crater is a cycle graph of length $k\ge 2$ or a single vertex with two loops (in which case we set $k=1$). Then the total degree of $X$ is $2kl^{d-1}$.
    If the crater is a single vertex with a single loop, the total degree of $X$ is 
    \[(l-1)^{-1}(l^{d+1}2-(l+1)).\]
    If the crater is a single vertex without any edges. Then the total degree is
    \[2(\ell+1)(\ell-1)^{-1}(\ell^d-1).\]
\end{lemma}
\begin{proof}For every level $i\ge 2$, we have
    \[\vert \{\textup{vertices of level $i$}\}\vert=\vert \{ \textup{vertices of level $i-1$}\}\vert\cdot l.\]
    Each vertex of level $i$ for $1\le i\le d-1$ has in and out degree $(l+1)$.
    Finally, every vertex of level $d$ has degree $1$.  If the crater is a cycle graph of length $k\ge 2$ or if the crater is a single vertex with two loops (so $k=1$), there are 
    $k$ vertices of level zero and $(l-1)k$ vertices of level $1$. 
    Thus, the total degree is
    \begin{align*}&\ (l+1)(k+k(l-1)+k(l-1)l+\dots +k(l-1)l^{d-2})+k(l-1)l^{d-1}\\
    =&\ (l+1)(k+k(l-1)(l^{d-1}-1)(l-1)^{-1})+k(l-1)l^{d-1}\\
    =&\ (l+1)kl^{d-1}+k(l-1)l^{d-1}=2kl^{d-1}.\end{align*}

    If the crater is a single vertex with one loop, there is one vertex of level zero and $l$ vertices of level $1$.
    We obtain a total degree of 
    \begin{align*}&\ (l+1)(1+l+\dots l^{d-1})+l^{d}\\=&\ (l+1)(l^{d}-1)(l-1)^{-1}+l^{d}\\
    =&\ (l-1)^{-1}(l^{d+1}2-(l+1)).\end{align*}

    If the crater is a single vertex without edges, there are $(l+1)$ vertices of level one and we obtain
    \begin{align*}&\ (l+1)(1+(l+1)+l(l+1)+\dots (l+1)l^{d-2})+(l+1)l^{d-1}\\
    =&\ (l+1)(1+(l+1)(l^{d-1}-1)(l-1)^{-1})+(l+1)l^{d-1}\\
    =&\ (\ell +1)(\ell-1)^{-1}((\ell -1)+(\ell +1)(\ell^{d-1}-1)+(\ell-1)\ell^{d-1})\\
    =&\ 2(\ell+1)(\ell-1)^{-1}(\ell^d-1).\end{align*}
\end{proof}
\begin{corollary}\label{cor:dobule-volcano}
    Let $X$ be an abstract $l$-volcano. Assume that the crater is a cycle of length $k$ or a single vertex with two loops. Let $\tilde{X}$ be the directed graph such that $\VV(\tilde X)=\VV(X)$ and $\EE(X)$ is defined by assigning an arbitrary orientation to each edge of $X$ and adding an additional edge in the opposite direction (except for loops).  Assume that $p$ is coprime to $2kl$. Then $\mu(\tilde X)=0$ and $\lambda(\tilde X)=1$.
\end{corollary}
\begin{proof}
Note that the number of spanning trees of $\tilde{X}$ is $k2^{k-1}$ if $k\ge 2$ and $2$ if $k=1$. In particular, $\kappa_X$ is not divisible by $p$ under our hypotheses. Furthermore, $\tilde{X}$ is balanced by construction. Lemma \ref{tot:degree} says that the total degree is $2kl^{d-1}$, which is also coprime to $p$. Thus, we can apply Theorem  \ref{thm:regular-balanced} to deduce that $\mu(\tilde{X})=0$ and $\lambda(\tilde{X})=1$.
\end{proof}

    The following remark details how Corollary~\ref{cor:dobule-volcano} can be applied to certain isogeny graphs of ordinary elliptic curves.

    \begin{remark}
        Let $\cY$ be a connected component of $Y_l^q(Np)$ of the graph defined as in Definition~\ref{def:intro}. We assume that the elliptic curves giving rise to $\cY$ are ordinary. The endomorphism rings of these elliptic curves are orders of an imaginary quadratic field $K$. We assume that $l$ is not inert in $K$. Recall that we assume $l\equiv 1\mod Np$. It follows from \cite[Theorem~5.5]{LM2} that the number of connected components of $Y_l^q(Np^n)$ lying above $\cY$ is bounded as $n$ increases. For each $n$, let $\cY_n$ be a connected component of $Y_l^q(Np^n)$ lying above $\cY$ such that the image of $\cY_{n+1}$ in $Y_l^    q(Np^n)$ is $\cY_n$. Then $(\cY_n)_{n\ge n_0}$ is a constant $\Zp$-tower over $\cY$ by Lemma~\ref{lem:constant}.

    It follows from \cite[Theorems~6.23 and 6.24]{LM2} that $\cY_n$ is a "finite tectonic volcano graph" for all $n\ge1$. Under the assumption that $l\equiv 1\mod Np$, $\cY$ is isomorphic to a graph $\tilde X$ obtained from  an abstract $l$-volcano after applying the procedure described in the statement of Corollary~\ref{cor:dobule-volcano}. In particular, if $p$ is coprime to $2$ and the length of the number of vertices in the crater of $\cY$, we have $\mu(\cY)=0$ and $\lambda(\cY)=1$.\label{rk:ordinary}
    \end{remark}
    

\bibliographystyle{alpha}
\bibliography{references}
\end{document}